\documentclass{amsart}

\usepackage{amsmath}
\usepackage{graphicx}
\usepackage{amsfonts}
\usepackage{amssymb}
\usepackage{mhequ}
\usepackage{Symbols}

\let\cal\mathcal
\def\eref#1{(\ref{#1})}
\def\E{\mathbf{E}}
\def\P{\mathbf{P}}
\def\one{\mathbf{1}}

\newtheorem{theorem}{Theorem}
\theoremstyle{plain}

\newtheorem{corollary}[theorem]{Corollary}

\newtheorem{proposition}[theorem]{Proposition}
\newtheorem{assumption}[theorem]{Assumption}
\newtheorem{remark}[theorem]{Remark}

\numberwithin{equation}{section}
\numberwithin{theorem}{section}

\begin{document}
\title{Ornstein-Uhlenbeck Processes on Lie Groups}
\author{Fabrice Baudoin, Martin Hairer, Josef Teichmann}
\address{Universit\'e Paul Sabatier, Institut des Math\'ematiques, 118, Rue de Narbonne, Toulouse, Cedex 31062, France; The University of Warwick, Mathematics Department, CV4 7AL Coventry, United Kingdom; Department of mathematical methods in economics, Vienna University of Technology, Wiedner Hauptstrasse 8--10, A-1040 Wien, Austria}
\email{fbaudoin@cict.fr, m.hairer@warwick.ac.uk, jteichma@fam.tuwien.ac.at}
\thanks{The first and third author gratefully acknowledge the support from the FWF-grant Y 328 (START prize from the Austrian Science Fund). The second author gratefully acknowledges the support from the EPSRC fellowship EP/D071593/1. All authors are grateful for the warm hospitality at the Mittag-Leffler Institute in Stockholm and at the Hausdorff Institute in Bonn.}
\begin{abstract}
We consider Ornstein-Uhlenbeck processes (OU-processes) associated to hypoelliptic diffusion processes on finite-dimensional Lie groups: let $ \mathcal{L} $ be a hypoelliptic, left-invariant ``sum of the squares''-operator on a Lie group $ G $ with associated Markov process $ X $, then we construct OU-processes by adding negative horizontal gradient drifts of functions $ U $.  In the natural case $ U(x) = - \log p(1,x) $, where $ p(1,x) $ is the density of the law of $ X $ starting at identity $ e $ at time $ t =1 $ with respect to the right-invariant Haar measure on $G$, we show the Poincar\'e inequality by applying the Driver-Melcher inequality for ``sum of the squares'' operators on Lie groups. The resulting Markov process is called the natural OU-process associated to the hypoelliptic diffusion on $ G $.

We prove the global strong existence of these OU-type processes on $ G $ under an integrability assumption on $U$. The Poincar\'e inequality for a large class of potentials $U$ is then shown by a perturbation technique. These results are applied to obtain a hypoelliptic equivalent of standard results on cooling schedules for simulated annealing on compact homogeneous spaces $M$.
\end{abstract}
\maketitle

\section{Introduction}

We consider an invariant hypo-elliptic diffusion process $ X $,
$$
dX^x_t = \sum_{i=1}^d V_i(X^x_t) \circ dB^i_t, \; X^x_0 = x
$$
on a connected Lie group $ G $ together with a right-invariant Haar measure $ \mu $, which is then also invariant for the diffusion process $ X $. The vector fields $ V_1,\ldots,V_d $ are assumed to be left-invariant vector fields, and their brackets generate the Lie algebra. In the spirit of sub-riemannian geometry the hypo-elliptic diffusion process can be used to define a metric (the Carnot-Carath\'eodory metric $d$) a geodesic structure and the notion of a gradient (the horizontal gradient $\operatorname{grad}_{hor}$) on $ G $. The density of the law of $ X^e_t $ with respect to $ \mu $ is denoted by $ p(t,.) $ and smooth by H\"ormander's theorem.

Fix $ t>0 $. The natural Ornstein-Uhlenbeck process on $G$ associated to $ X $ will be
$$
dY_t = \frac{1}{2} \operatorname{grad}_{hor} p(\tau,Y_t) dt + \sum_{i=1}^d V_i(Y_t) \circ dB^i_t
$$
out of two reasons: first the law of $ X^e_t $ is an invariant measure for $ Y $ and, second, there is a spectral gap for $ Y $ if $ X $ satisfies a Driver-Melcher inequality (see section 4 for precise details). We can easily extend all results to compact homogeneous spaces $M$.

The possible exponential convergence rate is then applied for new simulated annealing algorithms. The interest in those new algorithms lies in the fact that there are less Brownian motions than space dimensions of the optimization problem involved, and that the stochastic differential equations might have a smaller complexity and is therefore easier to evaluate. 

For this purpose we consider cases where the size of the spectral gap for the previously introduced process $Y$ is $ \frac{1}{2K \tau} $ for some constant $ K > 0 $ (see section 5 for precise details). This holds true for instance on $ SU(2) $ or on the Heisenberg tori. Let $ U $ be a smooth potential on $M$. We regard an equation of the type
$$
dZ_t = - \frac{1}{2} \operatorname{grad}_{hor} U(Z_t) dt + \eps \sum_{i=1}^d V_i(Z_t) \circ dB^i_t
$$
as a compact perturbation of the natural OU process in order to get an estimate for its spectral gap, which can be expressed by
$$
| U(x) + \eps^2 \log p(\eps^2,x_0,x) | \leq D
$$
for all $ x \in M $ and some constant $ D $. This amounts to a comparison of $ U $ with the square of the Carnot-Carath\'eodory metric $ {d(x_0,x)} $ due to short-time asymptotics of the heat kernel on the compact manifold $M$.

A concatenation of trajectories of such equations under a cooling schedule $ t \mapsto \eps(t)= \frac{c}{\sqrt{\log(R+t)}} $ leads then to the desired simulated annealing algorithms. The analytical arguments are strongly inspired by the seminal works of R.~Holley, S.~Kusuoka and D.~Stroock \cite{HolStr88} and \cite{HolKusStr89}.

\section{Preparations from functional analysis}

We consider a finite-dimensional, connected Lie group $ G $ with Lie algebra $ \mathfrak{g} $, its right-invariant Haar measure $ \mu $ and a family of left-invariant vector fields $ V_1,\ldots,V_d \in \mathfrak{g} $. We assume that H\"or\-mander's condition holds, i.e.~ the sub-algebra generated by $ V_1,\ldots,V_d $ coincides with $ \mathfrak{g} $.

We consider furthermore a stochastic basis $ (\Omega,\mathcal{F},\P) $ with a $d$-dimensional Brownian motion $ B $ and the Lie group valued process
\begin{equ}[e:SDEflat]
dX^x_t = \sum_{i=1}^d V_i(X^x_t)\circ dB^i_t, \quad X^x_0=x \in G\;.
\end{equ}
The generator of this process is denoted by $ \mathcal{L} $, we have
$$
\mathcal{L} = \frac{1}{2} \sum_{i=1}^d V_i^2\;,
$$
where we interpret the vector fields as first order differential operators on $ C^{\infty}(G,\mathbb{R}) $. Furthermore, we define a semigroup $P_t$ acting on bounded measurable functions  $ f:G \to \mathbb{R} $ by
$$
P_t f(x) = \E(f(X^x_t))\;.
$$
This semigroup can be extended to a strongly continuous semigroup on $L^2(G,\mu)$, which we will denote by the same letter $ P_t $. The carr\'e du champ operator $ \Gamma $ is defined for functions $ f $, where it makes sense, by
\begin{equ}[e:defChamp]
\Gamma(f,g) = \mathcal{L} (fg) - f \mathcal{L}g - g\mathcal{L}f\;.
\end{equ}
In our particular case, we obtain immediately
$$
\Gamma(f,f) = \sum_{i=1}^d {(V_if)}^2\;.
$$
Notice that the carr\'e du champ operator does not change if we add a drift to the generator $ \mathcal{L} $.

Due to the right invariance of the Haar measure $ \mu $ and the left-invariance of the vector fields $V_i$, the operator $ \mathcal{L} $ is symmetric (reversible) with respect to $ \mu $ and therefore $ \mu $ is an infinitesimal invariant measure in the sense that $\int \CL f(x) \, \mu(dx) = 0$ for all smooth compactly supported test functions $f$. Furthermore, due to the symmetry of $\mathcal L$ we have from \eref{e:defChamp} the relation
\begin{equ}[e:symL]
2 \int f \mathcal{L} g \mu = - \int \Gamma(f,g) \mu
\end{equ}
for all $ f \in \CC_0^\infty(G)$ and $ g \in \CC^{\infty}(G)$.

Let now $ U : G \to \mathbb{R} $ be an arbitrary smooth function such that
$$
Z:= \int_G \exp(-U(x)) \mu(dx) < \infty\;.
$$
We consider the modified generator
$$
\mathcal{L}^U := \mathcal{L} - \frac{1}{2}\Gamma(U,\cdot\,)\;,
$$
Notice that $\mu^U = \exp(-U) \mu $ is an infinitesimal invariant (finite) measure for $ \mathcal{L}^U $, since,
by \eref{e:symL},
\begin{equs}
\int \mathcal{L}^U f  \mu^U &= \int (\mathcal{L} f) \exp(-U) \mu - \frac{1}{2}\int \Gamma (U,f) \exp(-U)\mu \\
&= - \frac{1}{2} \int \Gamma (f, \exp(-U)) \mu - \frac{1}{2}\int \Gamma (f,U) \exp(-U)\mu = 0\;,
\end{equs}
and notice that $ \mathcal{L}^U $ is symmetric on $ L^2(\mu^U) $ in the sense that
\begin{equs}\eref{e:gap}
\int f \mathcal{L}^U g \mu^U = \int g \mathcal{L}^U f \mu^U = - \frac{1}{2} \int \Gamma(f,g) \mu^U
\end{equs}
for all smooth compactly supported test functions $ f,g: G \to \mathbb{R} $. The last equality is often referred to as integration by parts.

By definition and by integration by parts the operator $ \mathcal{L}^U $ has a spectral gap at $ 0 $ of size $ a > 0 $ in $L^2(\mu^U) $
if and only if
$$
\int_G \Gamma(f,f)(x)  \mu^U(dx) \geq 2 a \int_G {f(x)}^2 \mu^U(dx)\;,
$$
for all compactly supported smooth functions $ f $ on $ G $ satisfying
$$
\int_G f(x) \mu^U(dx) = 0\;.
$$

If we want to write an inequality for all test functions $ f $, it reads like
\begin{equ}[e:gap]
\int_G \Gamma(f,f)(x) \mu^U(dx)  \geq
 2a \Bigl(\int_G {f(x)}^2 \mu^U(dx) \int_G \mu^U(dx) - (\int_G f(x) \mu^U(dx))^2 \Bigr)
\end{equ}
for all test functions $ f \in \CC_0^\infty(G)$.

\section{Strong existence of OU-processes with values in Lie groups}

Let $G$ be a finite-dimensional, connected Lie group. We consider now the special case of the `potential' $ W_t(x) = - \log p(t,x) $, $ t > 0$, where
$p(t,x)$ is the density of the law of $X_t^e$ with respect to $\mu$.
By H\"ormander's Theorem \cite{HorArt,HorBook}, the function $(t,x) \mapsto p(t,x) $ is a positive and smooth function, hence the potential $ W_t $ is as
in the previous section. We write for short $ \mathcal{L}^t $ instead of $ \mathcal{L}^{W_t} $ and we call the associated Markov process the \textbf{natural OU-process on $G$ associated to the diffusion $X$}. We show that we have in fact global strong solutions for the corresponding Stratonovich SDE with values in $ G $. The next proposition is slightly more general.

\begin{proposition}\label{globalexistence}
Consider a smooth potential $ U : G \to \mathbb{R} $ such that
\begin{equ}
\int \exp(-U(x))\,\mu(dx) < \infty\;.
\end{equ}
Consider the following Stratonovich SDE with values in $ G $:
\begin{equ}[e:SDE]
dY^y_t = V_0(Y^y_t) dt + \sum_{i=1}^d V_i(Y^y_t)\circ dB^i_t\;, \quad Y^y_0=y \in G\;,
\end{equ}
where $ V_0 f = -\frac{1}{2}\Gamma(U,f) $ for smooth test functions $ f $.
Then there is a global strong solution to \eref{e:SDE} for all initial values $ y \in G $.
\end{proposition}

\begin{proof}
Since the coefficients defining \eref{e:SDE} are smooth by assumption, there exists a
unique strong solution up to the explosion time
\begin{equ}
\zeta_y = \inf \{t\,:\, \lim_{\tau \to t} Y_\tau^y = \infty\}\;.
\end{equ}
We then define a semigroup $\CP_t$ on $L^2(\mu^U)$ by
\begin{equ}[e:defPt]
\bigl(\CP_t f\bigr)(y) = \E \bigl(f(Y_t^y) \one_{\zeta_y > t} \bigr)\;.
\end{equ}
It can be shown in the exact same way as in \cite{Chern73,xuemeiThesis} that $\CP_t$ is a strongly continuous
contraction semigroup and that its generator $\CA$ coincides with $\CL^U$ on the set
$\CC_0^\infty(G)$ of compactly supported smooth functions.

On the other hand, setting $\CD(\CL^U) = \CC_0^\infty(G)$, one can show as in \cite{Chern73,xuemeiThesis} that $\CL^U$ is essentially self-adjoint, so that one must have $\CA = \overline {\CL^U} = (\CL^U)^*$.
In particular, since the constant function $\one$ belongs to $L^2(G, \mu^U)$ by the integrability
of $\exp(-U)$ and since $\int \bigl(\CL^U \psi\bigr)(x) \,\mu^U(dx) = 0$ for any test function
$\psi \in \CC_0^\infty(G)$, $\one$ belongs to the domain of $(\CL^U)^*$ and
therefore also to the domain of $\CA$. This then implies that $\CP_t \one = \one$ by the same
argument as in \cite{xuemeiThesis}. In particular, coming back to the definition \eref{e:defPt} of $\CP_t$,
we see that $\P(\zeta_y = \infty) = 1$ for every $y$, which is precisely the non-explosion
result that we were looking for.
\end{proof}

\begin{remark}
While this argument shows that, given a fixed initial condition $y$, there exists a unique global
strong solution $Y_t^y$ to \eref{e:SDE}, it does not prevent more subtle kinds of explosions,
see for example \cite{DavidCounterexample}.
\end{remark}

By Proposition \ref{globalexistence} and since $p(t,x)$ is smooth and integrable, it follows immediately that the OU-process exists globally in a strong sense.

\begin{corollary}
For any given $\tau>0$, the process
$$
dY^y_t = V_0(Y^y_t) dt + \sum_{i=1}^d V_i(Y^y_t)\circ dB^i_t, \quad Y^y_0=y \in G,
$$
with $ V_0 f = - \frac{1}{2} \Gamma(W_\tau,f) $ has a global strong solution.
\end{corollary}

\begin{remark}
More traditional Lyapunov-function based techniques seem to be highly non-trivial to apply for this
situation, due to the lack of information on the behaviour of $U(y)$ at large $y$. In view of
\cite{BA88CL,lea87a,lea87b}, it is tempting to conjecture that one has the asymptotic
\begin{equ}[e:asympt]
\lim_{\tau \to 0} {\tau}^2 \d_{\tau} \log p(\tau,y) = d^2(e,y)\;,
\end{equ}
and that the limit holds uniformly over compact sets $K$ that do not contain the origin $e$.
(Note that it follows from \cite{BA88CL} that this is true provided that $K$ does not intersect the
cut-locus.) If it were the case that \eref{e:asympt} holds, possible space-time scaling properties of
$p(\tau,x)$ could imply that, for every $\tau >0$, there exists a compact set $K$ such that
$\CL p(\tau,x) = \d_{\tau} p(\tau,x) > 0$ for $x \not \in K$. On the other hand, one has
$$
\CL^{\tau} W_{\tau} = -\frac{1}{2} \Gamma(W_{\tau},W_{\tau}) + \mathcal{L}W_{\tau} = - \frac{\mathcal{L}p(\tau,.)}{p(\tau,.)}\;,
$$
so that this would imply that $W_{\tau}$ is a Lyapunov function for the corresponding OU-process leading to another proof of the previous corollary.
\end{remark}

\section{Spectral Gaps for natural OU-processes}

Next we consider the question if $ \mathcal{L}^t $ admits a spectral gap on $ L^2(p_t \, \mu) $ for $ t > 0 $, which turns out to be a consequence of the Driver-Melcher inequality (see \cite{mel:04}).

\begin{theorem}\label{basic_equivalence}
The following assertions are equivalent:
\begin{itemize}
\item The operator $ \mathcal{L}^t $ has a spectral gap of size $ a_t > 0 $ on $ L^2(p_t \, \mu) $ for all $ t > 0 $, and a positive function $a:\mathbb{R}_{>0} \to \mathbb{R}_{>0} $.
\item The local estimate
$$
P_t (\Gamma(f,f))(g) \geq 2a_t ((P_t f^2)(g) - {((P_tf)(g))}^2)
$$
holds true for all test functions $ f :G \to \mathbb{R}$, for all $ t > 0$ and a positive function $a:\mathbb{R}_{>0} \to \mathbb{R}_{>0} $ at one (and therefore all) point $ g \in G $.
\end{itemize}
Furthermore, if we know that
$$
\Gamma(P_tf,P_tf)(e) \leq \phi(t) P_t(\Gamma(f,f))(e)
$$
holds true for all test functions $ f \in \CC_0^\infty(G)$, all $ t \geq 0 $, and a strictly positive locally integrable function $ \phi:\mathbb{R}_{\geq 0} \to \mathbb{R}_{> 0} $, then we can choose $a_t$ by
$$
a_t \int_0^t \phi(t-s)ds = \frac{1}{2}\;,
$$
for $ t > 0$ and the two equivalent assertions hold true.
\end{theorem}
\begin{proof}
Since $\mu^{W_t}$ is equal to the law of $X_t^e$, one has
$\int f \mu^{W_t} = P_t f (e)$ for every $f \in \CC_0^\infty(G)$.
The equivalence of the first two statements then follows from \eref{e:gap} and the fact that the translation invariance 
of \eref{e:SDEflat} implies
that if the bound holds at some $g$, it must hold for all $g \in G$. We fix a test function $ f :G \to \mathbb{R} $ as well as $ t > 0 $ and consider
$$
H(s) = P_s ((P_{t-s} f)^2)
$$
for $ 0 \leq s \leq t $. The derivative of this function equals
$$
H'(s)=P_s(\Gamma(P_{t-s}f,P_{t-s}f))
$$
and therefore -- assuming the third statement -- we obtain
$$
H'(s) \leq \phi(t-s) P_t (\Gamma(f,f)).
$$
Whence we can conclude
$$
H(t) - H(0) \leq \int_0^t \phi(t-s) ds P_t(\Gamma(f,f)),
$$
which is the second of the two equivalent assertions for an appropriately chosen $ a $.
\end{proof}
\begin{remark}
We can replace the Lie group $G $ by a general manifold $ M $, on which we are given a hypo-elliptic, reversible diffusion $ X $ with ``sum of the squares'' generator $ \mathcal{L} $. Then the analogous statement holds, in particular local Poincar\'e inequalities on $ M $ for $ \mathcal{L} $ lead to a spectral gap for the OU-type process $ \mathcal{L}^t $ with $ t > 0 $.
\end{remark}
\begin{corollary}\label{cor:SGnilpotent}
Let $ G $ be a free, nilpotent Lie group with $ d $ generators $ e_1,\ldots,e_d $ of step $m \geq 1$, and consider
$$
\mathcal{L} = \frac{1}{2} \sum_{i=1}^d e_i^2,
$$
then the operator $ \mathcal{L}^t $ has a spectral gap of size $ a_t = \frac{1}{2Kt} $ on $ L^2(p_t \, \mu) $ for some constant $K > 0$.
\end{corollary}
\begin{proof}
Due to the results of the very interesting Ph.D.-thesis \cite{mel:04} (see also \cite{drimel:07}), there is a constant $ K $ such that the bound
$$
\Gamma(P_tf,P_tf)(e) \leq K P_t(\Gamma(f,f))(e)
$$
holds true for all test functions $ f \in \CC_0^\infty(G) $ and for all times $ t \geq 0 $. This shows that $ a_t K t = \frac{1}{2}$, due to the assertions of Theorem \ref{basic_equivalence}.
\end{proof}
\begin{corollary}\label{cor:su2}
Let $ G = SU(2) $ be the Lie group of unitary matrices on $ \mathbb{C}^2 $ with Lie algebra $ \mathfrak{g} = \mathfrak{su}(2) = {\langle i,j,k \rangle}_{\mathbb{R}} $ with the usual commutation relations, i.e.~ $ [i,j] = 2k $ and its cyclic variants. Consider
$$
\mathcal{L} = \frac{1}{2}(i^2 + j^2),
$$
where we understand the elements $ i,j $ as left-invariant vector fields on $G$, then $ \mathcal{L}^t $ has a spectral gap of size
$ a_t = \frac{1}{2Kt} $ on $ L^2(p_t \, \mu) $ for some constant $K > 0$.
\end{corollary}
\begin{proof}
The result follows from \cite{baubon:08}[Prop. 4.20].
\end{proof}

\subsection{Generalisation to homogeneous spaces}
\label{sec:Homogeneous}

Consider now  $ M $ a compact homogeneous space with respect to the Lie group $ G $, i.e.~we have a (right) transitive action
$\hat \pi \colon G \times M \to M$ of $G$ on $M$. We assume that there exists a measure $ \mu^M $ on $ M $ which is invariant with respect to this action. We also assume that we are given a family $ V_1,\ldots,V_d $ of left-invariant vector fields on $G$
that generate its entire Lie algebra $ \mathfrak{g} $ as before. These vector fields induce fundamental vector fields $V_i^M $ on $ M $ by means of the action $\hat \pi$.

By choosing an `origin' $o \in M$, we obtain a surjection $\pi \colon G \to M$ by $\pi(g) = \hat \pi(g,o)$. The vector fields $ V_1,\ldots,V_d $ and $ V_1^M,\ldots,V_d^M $ are consequently $\pi$-related. Due to the invariance of $\mu^M$ with respect to the action $\hat \pi$, the vector fields $ V_i^M $ are anti-symmetric operators on $ L^2(\mu^M) $ and the generator
$$
\mathcal{L}^M = \frac{1}{2} \sum_{i=1}^d {(V_i^M)}^2 
$$
is consequently symmetric on $ L^2(\mu^M) $. In particular we have
$$
(V_i^M f)\circ \pi = V_i(f \circ \pi)
$$
for $ i =1,\ldots,d$. The local Driver-Melcher inequality on $ G $ translates to the same inequality on $ M $ by means of
$$
P^M_t (f)\circ \pi  = P_t(f \circ \pi) 
$$
for test functions $ f: M \to \mathbb{R} $, hence we obtain the corresponding Driver-Melcher inequality on $ M $ with the same constants, too.

\section{A simple result on simulated annealing}

By comparison with natural OU-processes on the homogeneous space $M$ we can obtain spectral gap results for quite general classes of potentials. For later purposes, namely for applications to simulated annealing algorithms, we shall state a parametrized version of a simple perturbation result.
\begin{theorem}\label{theo:pertPot}
Let $M$ be a homogeneous space. Let $ V_{\varepsilon} :M \to \mathbb{R} $ be a family of potentials $ V_{\eps} $ with
$$
|V_{\varepsilon} + \log p(\varepsilon,\cdot\,)| \leq  D_\eps
$$
for $ 0 < \varepsilon \leq 1 $, and constants $ D_\eps > 0 $. Assume furthermore that a Poincar\'e inequality holds for $ \mathcal{L}^{\varepsilon} $, i.e.
\begin{equ}[e:Poincareassum]
a_\eps P_{\varepsilon}(f^2)(e)  \leq P_{\varepsilon} (\Gamma(f,f))(e)
\end{equ}
for test functions $ f \in \CC_0^\infty(M)$ with $ P_{\varepsilon} f(e) = 0 $ and some constant $ a_{\varepsilon} >0 $ and $ 0 < \varepsilon \leq 1 $. Then one has $ \exp(-V_{\eps}) \in L^1(\mu) $ and the Poincar\'e inequality
\begin{equ}[e:Poincareconcl]
\int f^2(x) \exp(-V_{\eps}(x)) \mu(dx) \leq C_\eps \int \Gamma(f,f)(x) \exp(-V_{\eps}(x)) \mu(dx)
\end{equ}
holds for all test functions $ f \in \CC_0^\infty(M)$ with $ \int f(x) \exp(-V_{\eps}(x)) \mu(dx) = 0 $ and some constant $ C_\eps = \frac{\exp(2D_\eps)}{a_\eps} > 0 $. In particular, this leads to a spectral gap of size at least 
$$ \frac{1}{C_\eps}= \frac{a_\eps}{\exp(2D_{\eps})} $$ for $ \mathcal{L}^{V_{\eps}} $.
\end{theorem}
\begin{proof}
It follows immediately from the inequality
$$ p(\varepsilon,x) = \exp(-V_{\eps}(x)) \exp(V_{\eps}(x)) p(\varepsilon,x) \geq \exp(-D_{\eps}) \exp(-V_{\eps}(x)) $$
that $\exp(-V_{\eps}) \in L^1(\mu)$. Furthermore,
$$ \exp(-V_{\eps}(x)) =  \frac{1}{p(\varepsilon,x)\exp(V_{\eps}(x))} p(\varepsilon,x) \geq \exp(-D_{\eps}) p(\varepsilon,x)$$
for all $ x \in M $ by assumption. Hence we deduce \eref{e:Poincareconcl} with $ C_\eps = \frac{\exp(2D_\eps)}{a_\eps} $
from \eref{e:Poincareassum}.
\end{proof}
\begin{remark}
See \cite{bakledwan:06} for results on unbounded perturbations.
\end{remark}

Throughout the remainder of this section we assume that $M$ is a compact homogeneous space with respect to a connected Lie group $G$. We consider the same structures as in Section~\ref{sec:Homogeneous} on $ M $, but we omit the index $ M $ on vector fields and measures in order to improve readability. We shall furthermore impose the following assumption on the spectral gap:
\begin{assumption}\label{ass:gap_M}
There is a constant $ K > 0 $ such that
$$
\Gamma(P_tf,P_tf)(e) \leq K P_t(\Gamma(f,f))(e)
$$
holds true for all test functions $ f \in \CC_0^\infty(M) $ and $ 0 \leq t \leq 1 $.
\end{assumption}

We prepare now a quantitative simulated annealing result under the previous Assumption \ref{ass:gap_M} on $M$ and follow closely the lines of \cite{HolStr88}. Let $ U: M \to \mathbb{R} $ be a smooth potential.
The idea is to search for minima of $ U $ by sampling the measure 
$$ \frac{1}{Z_{\eps}}\exp(-\frac{U}{\eps^2}) \mu. $$
Recall that we do always assume that $ Z_{\eps} = \int_M \exp(-\frac{U}{\eps^2}) \mu < \infty $.
Sampling this measure is performed by looking at the invariant measure of
$$
\mathcal{L}_{\varepsilon} =  \mathcal{L} - \frac{1}{2} \Gamma\bigl(\frac{U}{\varepsilon^2},\cdot\,\bigr).
$$
Notice that the previous operator satisfies
$$
\varepsilon^2 \mathcal{L}_{\varepsilon} = \varepsilon^2 \mathcal{L} - \frac{1}{2} \Gamma(U,.),
$$
and a spectral gap for $ \varepsilon^2 \mathcal{L}_{\varepsilon} $ is a spectral gap for the diffusion process
$$
dY^y_t = V_0(Y^y_t) dt + \sum_{i=1}^d \varepsilon V_i(Y^y_t)\circ dB^i_t, \quad Y^y_0=y \in G,
$$
with $ V_0 f = - \frac{1}{2} \Gamma(U,f) $. Consequently we know -- given strong existence -- that the law of $ Y^y_t $ converges to $ \frac{1}{Z_{\eps}}\exp(-\frac{U}{\varepsilon^2}) \mu $ and concentrates therefore around the minima of $ U $. In this consideration $ \eps $ is considered to be fixed. Next we try to obtain a time-dependent version of the previous considerations, leading to a process concentrating precisely at the minima of $ U $.

In the following theorem we quantify the speed of convergence towards the invariant measure. We denote by $\mu_{\eps}$ the probability measure invariant for $\CL_{\eps}$ and we use the notation
$$
\operatorname{var}_{\varepsilon} (f)= {\langle{(f - {\langle f \rangle }_\varepsilon)}^2\rangle }_{\varepsilon}
$$
with $\scal{f}_\eps = \int_M f(g)\, \mu_{\eps}(dg)$ for the variance with respect to this measure. First we estimate the spectral gap along a cooling schedule $ t \mapsto \eps(t) $, then we prove that the measure concentrates around the minima of $U$ even in a time-dependent setting.

\begin{theorem}\label{dirichlet-form-estimate}
Let $ U :M \to \mathbb{R} $ be a smooth functions, $ D $ a constant, $ x_0 \in M $ a point such that
$$
|U(x) + \eps^2 \log p(\eps^2,x_0,x)| \leq D\;,
$$
for all $ x \in M $. Then there exist constants $ R,c > 0$ such that for $ \varepsilon(t) = \frac{c}{\sqrt{\log(R+t)}} $
$$
\operatorname{var}_{\varepsilon(t)}(f) \leq K{(R+t)} \scal{\frac{1}{2}\Gamma(f,f)}_{\eps(t)}\;,
$$
for all test functions $ f \in \CC_0^\infty(M)$ and $ t \geq 0 $.
\end{theorem}
\begin{proof}
We can start to collect results. Combining Assumption~\ref{ass:gap_M} with Theorem~\ref{theo:pertPot} applied for $ V_\eps = \frac{U}{\eps^2} $, we obtain that spectral gap for the operator $ \mathcal{L}_{\varepsilon} $ has size
at least
$$
\frac{1}{K \varepsilon^2} \exp\Bigl(-\frac{2D}{\varepsilon^2}\Bigr)
$$
for $ 0 < \varepsilon \leq 1 $, so that $ \varepsilon^2 \mathcal{L}_{\varepsilon} $ has a spectral gap of size at least
$$
\frac{1}{K} \exp\Bigl(-\frac{2D}{\varepsilon^2}\Bigr).
$$
We choose $ c^2 = 2D$ and $ R $ sufficiently large so that $ \varepsilon(t) \leq 1 $ for $ t \geq 0 $, and we conclude that
$$
K \exp\Bigl(\frac{2D}{\varepsilon(t)^2} \Bigr) \leq K {(R + t)}
$$
for all $ t \geq 0 $, which yields the desired result.
\end{proof}

We denote by $ Z $ the process with cooling schedule $ t \mapsto \varepsilon(t) $ as in the previous theorem,
$$
dZ^z_t = V_0(Z^z_t) dt + \sum_{i=1}^d \varepsilon(t) V_i(Z_t^z) \circ dB^i_t\;,
$$
where the drift vector field is given through $ V_0 f = - \frac{1}{2} \Gamma(U,f) $. Then the previous conclusion leads to the following proposition.
\begin{proposition}
Let $ f_t $ denote the Radon-Nikodym derivative of the law of $Z_t^z$ with respect to 
$ \mu_{\eps(t)} $ and let $$ u(t) := {\|f_t - 1 \|}^2_{L^2(\mu_{\eps(t)})} $$ denote its distance in $L^2(\mu_{\eps(t)})$ to $1$ (which corresponds to $\operatorname{var}_{\eps(t)}(f_t) $), then
$$
u'(t) \leq - \frac{1}{K(R+t)}u(t) + \frac{N}{c^2(R+t)} u(t) + \frac{2N}{c^2(R+t)} \sqrt{u(t)}
$$
for the constants $R$,$c$ and $K$ from Theorem \ref{dirichlet-form-estimate}, and $ N = \max U - \min U $.
\end{proposition}

\begin{remark}\label{cKestimate}
We find $ c^2 > 3NK $, such that $ sup_{t \geq 0} u(t) $ is bounded from above by a constant depending on $ f_0 $, $c$, $R$. $N$ and $K$.
\end{remark}

\begin{proof}
The proof follows closely the lines of \cite{HolStr88}. By assumption we know that $ \operatorname{var}_{\eps(t)}(f_t) = u(t) = {||{f_t}||}^2_{L^2(\mu_{\eps(t)})} -1 $ and hence with the notation $ \beta(t) =  \frac{1}{{\varepsilon(t)}^2} $,
\begin{equs}
u'(t) &= - \scal{\Gamma(f_t,f_t)}_{\eps(t)} - \beta'(t) \int (U-{\langle U \rangle}_{\varepsilon(t)}) f_t^2 \mu_{t} \\ 
&= - \scal{\Gamma(f_t,f_t)}_{\eps(t)} - \beta'(t) \int (U-{\langle U \rangle}_{\varepsilon(t)}) {(f_t-1)}^2 \mu_{t} - \\
 &\quad - 2\beta'(t) \int (U-{\langle U \rangle}_{\varepsilon(t)}) (f_t-1) \mu_{t} \\
&\leq  - \frac{1}{K(R+t)}u(t) + \frac{N}{c^2(R+t)} u(t) + \frac{2N}{c^2(R+t)} \sqrt{u(t)}\;.
\end{equs}
Here, we used the Cauchy-Schwarz inequality and the conclusions of the previous Theorem \ref{dirichlet-form-estimate} in the last line.
\end{proof}

\begin{assumption}\label{ass:growth}
Let $ U : M \to \mathbb{R} $ be a potential on $M$ such that
$$
| U(x) - d(x,x_0)^2 | \leq D_1
$$
for some positive constant $ D_1 $, a point $ x_0 \in M $ and for all $ x \in M $. Here we denote by $ d(x,x_0)$ the Carnot-Carath\'eodory metric on $M$, see for instance \cite{lea87a,lea87b} and \cite{San84}.
\end{assumption}

\begin{remark} For non-compact manifolds $ M $ the limit
$$
\lim_{t \to 0} t \log p(t,x_0,x) = - {d(x_0,x)}^2
$$
is uniform on compact subsets of $M$, but usually not on the whole of $M$. An abelian, non-compact example where the limit is globally uniform is $ M=\mathbb{R}^d $. On the simplest non-compact and non-abelian example, the Heisenberg group $ G_d^2 $, the limit is not uniform, see recent work of H.-Q.~Li \cite{Li07}. Therefore we consider in our perturbation argument only compact manifolds $M$.
\end{remark}

On compact manifolds $M$ we know due to R.~ L\'eandre's beautiful results \cite{lea87a,lea87b} that
there is $ D_2 $ such that
$$
|{d(x_0,x)}^2+ \eps^2 \log p(\eps^2,x_0,x)| \leq D_2
$$
for $x \in M $. Hence we can conclude by the triangle inequality that the potential $ U $ satisfies the assumptions of Theorem \ref{dirichlet-form-estimate}.

\begin{theorem}
Assume Assumptions \ref{ass:gap_M} and \ref{ass:growth}, and assume that 
$$ sup_{t \geq 0} \|f_t\|_{L^2(\mu_{\eps(t)})} < \infty \; .$$
Define $U_0 = \inf_{x\in M}U(x)$ and, for every $\delta> 0$, denote by $A_\delta$ the set
$ A_{\delta} = \{x\in M\,| \, U(x) \geq U_0 + \delta\} $. Then we can conclude that
$$
\P\bigl(Z_t^z \in A_{\delta}\bigr) \leq M \sqrt{\mu_{\eps(t)}(A_{\delta})}
$$
for every $t > 0$ and every $ \delta \geq 0 $.
\end{theorem}
\begin{proof} It follows from the Cauchy-Schwarz inequality that
$$
\P\bigl(Z_t^z \in A_{\delta}\bigr) = \int_{A_{\delta}} f_t \mu_{\eps(t)} \leq M \sqrt{\mu_{\eps(t)}(A_{\delta})},
$$
as required.
\end{proof}
\begin{remark}
Since $\lim_{\eps \to 0} \mu_\eps(A_{\delta}) = 0$ for every $\delta > 0$, we obtain that for all continuous bounded test functions $ f $, we have
$$
\E(f(Z_t^z)) \rightarrow f(x_0)\;,
$$
provided that there is only one element $ x_0 \in M $ such that $U(x_0) = U_0$.
\end{remark}
\begin{remark}
We can improve the previous result from $L^2$-estimates to $L^q$-estimates for $ {|| f_t - 1 ||}_{L^q(\mu_t)} $, for $ q > 2 $: this follows \cite[Theorem 2.2]{HolKusStr89} and the fact that the proof of \cite[Theorem 2.7]{HolKusStr89} applies due to a valid Sobolev inequality, i.e.~for every $ q > 2 $ there is a constant $ C_0 > 0 $ such that
$$
{|| f ||}^2_{L^q(\mu)} \leq C_0 \big( \frac{1}{2}\int_M \Gamma(f,f) \mu + {|| f ||}^2_{L^2(\mu)} \big)
$$
holds for all test functions $ f: M \to \mathbb{R} $. Such a Sobolev inequality can be found for instance in \cite[Section 3]{HorArt} or in \cite{San84}.
\end{remark}

\begin{remark}
We can apply hypo-elliptic simulated annealing to potentials on compact nilmanifolds with the respective subriemannian structure due to Corollary \ref{cor:SGnilpotent}, or we can apply it to potentials on $ SU(2) $ due to Corollary \ref{cor:su2}.

The implementation of those new algorithms can yield some advantages with respect to elliptic simulated annealing algorithms on Riemannian manifolds as described in \cite{HolKusStr89}. On the one hand the number of Brownian motions involved is smaller, such as the complexity of the SDE
$$
dZ^z_t = V_0(Z^z_t) dt + \sum_{i=1}^d \varepsilon(t) V_i(Z_t^z) \circ dB^i_t\;.
$$
as a whole, since less vector fields have to evaluated and the gradient $ V_0 $ is less complex being a horizontal gradient. The price to pay is a larger constant $ c > 0 $ in the rate of convergence. In cases where hypo-elliptic simulated annealing can be directly compared with elliptic simulated annealing on flat space (for instance on $3$-torus $ \mathbb{T}^3 $, where we have the flat euclidean structure and the subriemmanian structure of the Heisenberg torus with two generators) the elliptic algorithm is superior. This is due to the fact the one can choose the vector fields in the elliptic algorithm constant (on flat $ \mathbb{T}^3 $) which reduces the complexity considerably, whereas one has to apply more sophisticated vector fields in the case of the Heisenberg group.

For optimization on $ SU(2) $, where our theory applies due to Corollary \ref{cor:su2}, we have a visible advantage over the elliptic simulated annealing, since we need one more Brownian motion and one more vector field for the elliptic algorithm, and we cannot simplify the vector fields in the elliptic algorithm.
\end{remark}


\end{document}